\documentclass[12pt,draft,reqno]{amsart}

\usepackage{amssymb} 
\usepackage{mathrsfs} 
\usepackage{stmaryrd} 
\usepackage{color}

\DeclareMathAlphabet{\mathpzc}{OT1}{pzc}{m}{it} 

\newtheorem{Thm}{Theorem}[section]
\newtheorem{Cor}[Thm]{Corollary}
\newtheorem{Lem}[Thm]{Lemma}
\newtheorem{Def}[Thm]{Definition}

\theoremstyle{definition}
\newtheorem{Rem}[Thm]{Remark}

\theoremstyle{definition}

\theoremstyle{definition}
\newtheorem{Example}[Thm]{Example}

\theoremstyle{definition} 

\makeatletter
\@addtoreset{equation}{section}
\makeatother

\newcommand{\x}{\mr{x}}

\newcommand{\eps}{\varepsilon}

\DeclareMathOperator{\sinc}{sinc}

\newcommand\alg{\mathbb{A}} 

\newcommand\funzione{\longrightarrow}
\newcommand\ci{\mathbb{C}} 
\newcommand\erre{\mathbb{R}} 

\newcommand{\su}{\mathbb{S}}  

\renewcommand\j{\mathbf{j}}  
\renewcommand\k{\mathbf{k}} 

\newcommand{\quat}{\mathbb{H}}

\DeclareMathOperator{\re}{Re} 
\DeclareMathOperator{\im}{Im} 

\newcommand\setmeno{\!\smallsetminus\!} 

\providecommand{\clint}[1]{\hspace{0.045ex}[#1]} 
\providecommand{\opint}[1]{\hspace{0.15ex}\left]#1\right[\hspace{0.15ex}} 

\newcommand\norma[2]{\Vert #1\Vert_{#2}} 

\newcommand{\A}{\mathsf{A}}  
\newcommand{\B}{\mathsf{B}}  
\newcommand{\Lap}{\mathsf{L}}  
\newcommand{\Id}{\ \!\mathsf{Id}}

\newcommand\End{\textsl{End}}
\renewcommand\r{\textsl{r}} 

\newcommand\Lin{\mathscr{L}}

\newcommand\s{\mathpzc{S}} 

\newcommand{\C}{\mathsf{C}} 
\newcommand{\Q}{\mathsf{Q}}

\newcommand{\T}{\mathsf{T}}  
\renewcommand{\S}{\mathsf{S}}  

\providecommand{\opint}[1]{\hspace{0.15ex}\left]#1\right[\hspace{0.15ex}} 
\providecommand{\clsxint}[1]{\hspace{0.1ex}\left[#1\right[\hspace{0.15ex}} 
\DeclareMathOperator{\de}{d \! \hspace{0.2ex}} 

\renewcommand{\S}{\mathsf{S}}  

\renewcommand\sp{\hspace{2.8ex}} 

\newcommand\enne{\mathbb{N}} 

\newcommand\X{\textsl{X}\hspace{0.1ex}}

\def\vuoto{\varnothing} 

\newcommand{\RR}{\mathbb{R}}  
\newcommand{\CC}{\mathbb{C}}

\newcommand{\mr}{\mathrm}

\newcommand{\mc}{\mathcal}

\definecolor{blu1}{rgb}{0.1,0.1,1}

\definecolor{verde}{rgb}{0.1,0.4,0.2}

\definecolor{pink}{rgb}{1.0, 0.33, 0.64}

\newcommand{\convstar}{\star}

\begin{document}


\title[On the generators of Clifford semigroups]{On the generators of Clifford semigroups: polynomial resolvents and their integral transforms}

\author{Riccardo Ghiloni and Vincenzo Recupero}

\thanks{R. Ghiloni is a member of GNSAGA-INDAM. V. Recupero is a member of GNAMPA-INDAM}

\address{\textbf{Riccardo Ghiloni}\\
        Dipartimento di Matematica\\ 
        Universit\`a di Trento\\
        Via Sommarive 14\\ 
        38123 Trento\\ 
        Italy. \newline
        {\rm E-mail address:}
        {\tt ghiloni@science.unitn.it}}
        
\address{\textbf{Vincenzo Recupero}\\
        Dipartimento di Scienze Matematiche\\ 
        Politecnico di Torino\\
        Cor\-so Duca degli Abruzzi 24\\ 
        10129 Torino\\ 
        Italy. \newline
        {\rm E-mail address:}
        {\tt vincenzo.recupero@polito.it}}

   
\date{}


\begin{abstract}
This paper deals with generators $\A$ of strongly continuous right linear semigroups in Banach two-sided spaces whose set of scalars is an arbitrary Clifford algebra $\mathit{C}\ell(0,n)$. We study the invertibility of operators of the form $P(\A)$, where $P(\x)\in\RR[\x]$ is any real polynomial, and we give an integral representation for $P(\A)^{-1}$ by means of a Laplace-type transform of the semigroup $\T(t)$ generated by $\A$. In particular, we deduce a new integral representation for the operator $(\A^2 - 2\re (q) \,\A + |q|^2)^{-1}$. As an immediate consequence, we also obtain a new proof of the well-known integral representation for the $S$-resolvent operator of $\A$ (also called spherical resolvent operator of $\A$).
\end{abstract}


\maketitle


\thispagestyle{empty}


\begin{flushright}
{\it dedicated to Professor Klaus G\"{u}rlebeck}
\end{flushright}

\vspace{1em}

\section{Introduction and main results}

Quaternionic functional analysis has probably its original motivation in the seminal paper \cite{BirNeu36}, where it is pointed out that quantum mechanics may be formulated, not only on complex Hilbert spaces, but also on Hilbert spaces whose set of scalars is $\quat$, the noncommutative real algebra of quaternions.

Many papers have been devoted to the development of quantum mechanics in the quaternionic framework (see, e.g., \cite{FinJauSchSpe62, Emc63, HorBie84, Adl95}), whose natural setting is a Hilbert two-sided $\quat$-module $X$, with the space of bounded linear operators acting on it replaced by the set $\Lin^\r(X)$ of bounded \emph{right} linear operators. 
However, a full development of quaternionic quantum mechanics was prevented by the lack of suitable quaternionic spectral notions, indeed, as observed in \cite{CoGeSaSt} (see also \cite{GhiMorPer13}), the classical definitions of spectrum and resolvent operator do not allow to define a noncommutative functional calculus. 

A rigorous formulation of a quaternionic spectral theory has been provided for the first time in \cite{CoGeSaSt07} where one can find the definition of the notions of \emph{$S$-resolvent set} 
$\rho_\s(\A)$,  \emph{$S$-resolvent operator} $S^{-1}(q,\A)$ at $q$, and \emph{$S$-spectrum} 
$\sigma_\s(\A)=\quat\setminus\rho_\s(\A)$ of a right linear operator $\A$ on a quaternionic Banach space $X$. They are given by
\begin{align}
  & \rho_\s(\A) := \{q \in \quat\ :\ \exists (\A^2 - 2\re (q) \,\A + |q|^2)^{-1} \in \Lin^\r(X)\}, \label{S1} \\
  & S^{-1}(q,\A) :=\Q_q(\A) \bar{q} - \A \Q_q(\A), \qquad q \in \rho_\s(\A),\label{S2}
\end{align}
where
\[
\Q_q(\A) := (\A^2 - 2\re (q) \,\A + |q|^2)^{-1}, \qquad q \in \rho_\s(\A).
\]
For a historical account of these notions we refer the reader to the Introduction of \cite{ColGan19}. We warn the reader that in the present paper, coherently with the notations we used in \cite{GhiRec16, GhiRec18}, we employ the more geometrically expressive terms \emph{spherical resolvent operator} and \emph{spherical spectrum} in place of $S$-resolvent operator and $S$-spectrum, respectively; moreover, we use the symbol $\C_q(\A)$ for the generalized $\C$auchy-like resolvent operator $S^{-1}(q,\A)$. Thus, given $q \in \rho_\s(\A)$, we have that $\C_q(\A)=S^{-1}(q,\A)=\Q_q(\A) \bar{q} - \A \Q_q(\A)$. In addition we call $\Q_q(\A)$ \emph{spherical quasi-resolvent operator of $\A$ at $q$.} Sometimes, in literature, $\Q_q(\A)$ is called pseudo-resolvent operator of $\A$ at $q$. However, in the classical complex case, the term ``pseudo-resolvent operator'' has a different meaning (cf., e.g., \cite[Section 1.9]{Paz83}). For this reason, we think that it is important to keep the term ``spherical pseudo-resolvent operator'' (or ``$S$-pseudo-resolvent operator'') for future possible generalizations of the complex notion of pseudo-resolvent operator to the hypercomplex setting.
 
The above definitions \eqref{S1} and \eqref{S2} permit to develop a noncommutative functional calculus for right linear operators on a Banach two-sided module over $\quat$ (and over a Clifford algebra as well, cf. \cite{ColSabStr08, ColSab09, ColSab10, CoGeSaSt, CoGeSaSt10, CoSaSt, GhiMorPer13}) and to deduce in \cite{ACK16} (see also \cite{GhiMorPer-bis}) the spectral representation theorems for normal operators in the quaternionic Hilbert setting. A wider bibliograpy can be found in the recent accounts on the theory \cite{ColGanKim18,ColGan19}.

The mentioned noncommutative functional calculus is intimately connected to the theory of slice regular functions, introduced in \cite{GeSt}, which extends to quaternions the classical concept of holomorphic function. They form a class of functions admitting a local power series expansion at every point of their domain of definition (cf. \cite{GenSto12}), including polynomials with quaternionic coefficients on one side, and they admit
a Cauchy-type integral representation formula with a suitable quaternionic version of the kernel
proved for the first time in \cite{ColSab2009} (see also \cite{CoGeSa,GhiPerRec17}).

The next natural stage in this analysis is the development of a noncommutative theory of right linear operator semigroups which was developed in \cite{ColSab11, ACS2015, GhiRec16, ACFGK2017, GhiRec18, ColGan19}. In the classical complex theory a fundamental tool is provided by the integral representation of the resolvent operator of a generator $\A$ by means of the Laplace transform of the semigroup $\T(t)$ generated by $\A$. An analogous integral representation in the quaternionic case for the spherical resolvent operator $\C_q(\Q)$ is shown in \cite{ColSab11} and a proof is provided in \cite{GhiRec18} using techniques from slice regular function theory.

The purpose of the present paper is to study the invertibility of operators of the form $P(\A)$ where $P(\x)$ is an arbitrary polynomial with real coefficients of degree at least~$2$, including $P(\x)=\Delta_q(\x)=\x^2-2\re(q)\x+|q|^2$. Under natural conditions, we prove that $P(\A)^{-1}$ exists and belongs to $\Lin^\r(X)$. Furthermore, we provide an integral representation for $P(\A)^{-1}$ by means of a Laplace-type transform of the semigroup $\T(t)$ generated by $\A$. We extend this integral representation to operators of the form $\sum_{j=0}^{d-1}\A^jP(\A)^{-1}p_j$, where $d$ is the degree of $P$ and $p_0,\ldots,p_{d-1}$ are arbitrarily chosen quaternions. In the case $P(\x)=\Delta_q(\x)$, we obtain a new integral representation for $\Q_q(\A)$ and, setting $p_0:=\overline{q}$ and $p_1:=-1$, we discover again the well-known integral representation for $\C_q(\A)$ via the quaternionic Laplace transform (see \cite{ColSab09,ColSab11}). This gives also a new proof of the integral representation for $\C_q(\A)$, which avoids the use of slice regular function techniques. Our results are valid not only on quaternions but also on a class of real associative $^*$-algebras including, as the main examples, all Clifford algebras $\mathit{C}\ell(0,n)$.

Let $n\in\enne$, let $\RR_n$ be the Clifford algebra $\mathit{C}\ell(0,n)$ equipped with the Clifford conjugation and the Clifford operator norm $|\cdot|$. Consider a Banach two-sided $\RR_n$-module $X$ with norm $\|\cdot\|$ and the set $\Lin^\r(X)$ of all bounded right linear operators on $X$ (all the precise de\-finitions will be recalled in the next section).

Let $m\in\enne$ and let $P(\x)=\sum_{k=0}^{m+2}\x^ka_k\in\RR[\x]$ be a polynomial with real coefficients in the indeterminate~$\x$. Suppose $P$ has degree $m+2$, that is, $a_{m+2}\neq0$. Given a right linear operator $\A:D(\A) \funzione X$, we define the right linear operator $P(\A):D(\A^{m+2}) \funzione X$ simply by replacing $\x$ with $\A$, that is, $P(\A):=\sum_{k=0}^{m+2}\A^ka_k$. Denote by $C^\infty(\clsxint{0,\infty};\RR)$ the set of all infinitely many times differentiable functions $g:\clsxint{0,\infty}\funzione\RR$. Consider the following ODE with constant coefficients in the variable $g\in C^\infty(\clsxint{0,\infty};\RR)$:
\begin{equation}\label{ODE}
\begin{cases}
P\left(-\frac{d}{dt}\right)(g)=0\; \text{ on $\clsxint{0,\infty}\,$,} \\
g(0)=g'(0)=\ldots=g^{(m)}(0)=0\,, \\
g^{(m+1)}(0)=(-1)^m(a_{m+2})^{-1}\,, 
\end{cases}
\end{equation}
where $g^{(k)}$ is the $k^{\mr{th}}$-derivative of $g$ and $P\left(-\frac{d}{dt}\right)(g):=\sum_{k=0}^{m+2}g^{(k)}(-1)^ka_k$. Denote by $g_P\in C^\infty(\clsxint{0,\infty};\RR)$ the unique solution of \eqref{ODE}. Recall that, if $\lambda_1,\ldots,\lambda_h$ are the complex roots of the polynomial $P(\x)$ with multiplicity $m_1,\ldots,m_h$, then there exist complex polynomials $Q_1(\x),\ldots,Q_h(\x)\in\CC[\x]$ such that the degree of each $Q_j(\x)$ is $<m_j$ and
\begin{equation}\label{g_P}
g_P(t)=\sum_{j=1}^hQ_j(t)e^{-\lambda_jt}.
\end{equation}  
Define $r_P\in\RR$ by
\begin{equation}
r_P:=\min\{\Re(\lambda_1),\ldots,\Re(\lambda_h)\},
\end{equation}
where $\Re(\lambda_j)$ is the real part of the complex number $\lambda_j$.

Recall that if $\T:\clsxint{0,\infty} \funzione \Lin^r(X)$ is a strongly continuous right linear semigroup then it is also a strongly continuous semigroup with $X$ considered as a real vector space, therefore there exists $\omega\in\RR$ such that $\sup_{t\in\clsxint{0,\infty}}\norma{\T(t)}{}e^{-\omega t}<\infty$ (see, e.g., \cite[Proposition I.5.5, p. 39]{EngNag00}).

Our main result reads as follows.

\begin{Thm}\label{thm:supermain-clifford}
Let $\T:\clsxint{0,\infty} \funzione \Lin^r(X)$ be a strongly continuous right linear semigroup, let $\A:D(\A) \funzione X$ be its generator and let $\omega\in\erre$ be a real constant such that $M:=\sup_{t\in\clsxint{0,\infty}}\norma{\T(t)}{}e^{-\omega t}<\infty$. Then, if $r_P>\omega$, the operator $P(\A)$ is bijective, $P(\A)^{-1}\in\Lin^r(X)$ and it holds:
 \begin{equation}\label{inverso P = integrale-clifford}
 P(\A)^{-1}x=\int_0^\infty\T(t)g_P(t)x \de t \qquad \forall x\in X
 \end{equation}
and
\begin{equation}\label{estimate for P-1-clifford}
 \norma{P(\A)^{-1}}{}  \le M\sum_{j=1}^h\sum_{k=1}^{m_j}\frac{|c_{-k}^{(j)}|}{(r_P-\omega)^k}\,,
 \end{equation}
where $c_{-k}^{(j)}$ is the residue at $\lambda_j$ of the meromorphic function $z\mapsto a_{m+2}(z-\lambda_j)^{k-1}/P(-z)$.

Furthermore, given $(p_0,p_1,\ldots,p_{m+1})\in(\RR_n)^{m+2}$, we have
 \begin{equation}\label{potA P-1-clifford}
 \sum_{j=0}^{m+1}\A^kP(\A)^{-1}p_jx=\int_0^\infty\T(t)\left(\sum_{j=0}^{m+1}g_P^{(j)}(-1)^jp_jx\right) \de t \qquad \forall x\in X.
 \end{equation}
\end{Thm}

Let $Q_{\RR_n}$ be the quadratic cone of $\RR_n$ (see \eqref{quadratic cone} for the definition) and let $q\in Q_{\RR_n}$. Define the function $g_q: \erre \funzione \erre$ by 
\begin{equation}\label{g_q-clifford}
g_q(t) := te^{-\re(q)t}\sinc(t|\im(q)|), \qquad t \in \erre,
\end{equation} 
where $\re(q)$ and $\im(q)$ are the real and imaginary parts of $q$, respectively. We recall that $\sinc : \erre \funzione \erre$ is the \emph{unnormalized sinc function}, that is, the real-valued continuous function $\xi$ on $\erre$ defined by $\xi(0)=1$ and $\xi(r) = \sin(r)/r$ for all $r \neq 0$. 

Thanks to the preceding result, we are able to prove the following:

\begin{Thm}\label{thm:main-clifford}
Let $\T:\clsxint{0,\infty} \funzione \Lin^r(X)$ be a strongly continuous right linear semigroup, let $\A:D(\A) \funzione X$ be its generator, and let $\omega\in\erre$ be a real constant such that 
$M:=\sup_{t\in\clsxint{0,\infty}}\norma{\T(t)}{}e^{-\omega t}<\infty$. Consider any $q\in Q_{\RR_n}$ and set $a:=\re(q)$ and $b:=|\im(q)|$. If $\re(q)>\omega$, then
we have that $q \in \rho_\s(\A)$ and it holds:
\begin{alignat}{3}
& \Q_q(\A)x = \int_0^\infty\T(t)g_q(t) x\de t =\int_0^\infty\T(t)te^{-ta}\sinc(tb) x\de t, \label{eq:Q_S}\\
& \A\Q_q(\A)x = -\int_0^\infty\T(t)g'_q(t) x\de t =
   -\int_0^\infty\T(t)e^{-ta}(\cos(tb) - at\sinc(tb)) x\de t, \label{eq:AQ_S}\\
& \C_q(\A)x=\int_0^\infty\T(t)e^{-tq}x\de t \label{eq:C_S}
\end{alignat}
for every $x \in X$.
Moreover, we have:
\begin{align}
&\norma{\Q_q(\A)}{} \leq \frac{M}{(\re(q)-\omega)^2} \,, \label{eq:estimateQ}\\
&\norma{\C_q(\A)}{} \leq \frac{M}{\re(q)-\omega}\,.\label{eq:estimateC}
\end{align}
\end{Thm}

Here is the plan of the paper. In the following section we recall all the needed precise definitions.  In Section 3 we present the preceding theorems in the more general case of certain real $*$-algebras, including all the $\RR_n$'s. Section \ref{S:proofs} is devoted to the proofs of these theorems. Finally, in Section \ref{Qq^n} we apply the main theorem in order to derive an integral representation of the integer powers of the spherical quasi-resolvent operator and the estimate of their norms; this extends to $\Q_q(\A)$ our Theorem 6.6 in \cite{GhiRec18} concerning the Laplace-type transform for the integer slice powers of $\C_q(\A)$.


\section{Preliminaries}\label{S:preliminaries}

Throughout all the paper we will assume that $\alg$ is a nontrivial finite dimensional $\erre$-vector space endowed with a bilinear product $\alg \times \alg \funzione \alg: (p,q) \longmapsto pq$ with unit $1_\alg$, and with a mapping
$\alg \funzione \alg : q \longmapsto q^c$ called \emph{$^*$-involution}, which is an $\erre$-linear mapping such that 
$(q^c)^c = q $,  $(pq)^c = q^c p^c$, and $r^c = r$ for all $p, q \in \alg, r \in \erre \subseteq \alg$, where we are identifying $\erre$ with the subalgebra of $\alg$ generated by $1_\alg$ by means of the algebra isomorphism 
$\erre \funzione \erre 1_\alg : r \longmapsto r1_\alg$. Therefore we can write $1 = 1_\alg$ and we summarize the 
previous assumptions by saying that
\begin{equation}\label{assumption on A}
\begin{split}
  & \text{$\alg$ is a finite dimensional associative nontrivial real $^*$-algebra} \\
  & \text{with $^*$-involution $q \longmapsto q^c$ and unit $1$.}
\end{split}
\end{equation}
Under the previous assumptions we define the \emph{imaginary sphere $\su_\alg$} in $\alg$ by
\begin{equation}\label{imaginary sphere}
  \su_\alg := \left\{q \in \alg\ :\ q^c = -q,\ q^2 = -1\right\}.
\end{equation}
In the remainder of the paper we will assume that
\begin{equation}\label{S_A nonvuota}
  \su_\alg \neq \varnothing.
\end{equation}
Condition \eqref{S_A nonvuota} in particular implies that $\alg$ cannot be equal to $\erre$. We set
\begin{align}
  & \ci_\j := \left\{r+s\j \in \alg\ :\ r, s \in \erre\right\}, \qquad 
        \j \in \su_\alg, \\
  & Q_\alg := \bigcup_{\j \in \su_\alg} \ci_\j, \label{quadratic cone}
\end{align}
the set $Q_\alg$ being called \emph{quadratic cone of $\alg$}. The \emph{real part $\re(q)$} and the 
\emph{imaginary part $\im(q)$} of an element $q \in \alg$ are defined by
\begin{equation}\label{real and imaginary parts}
  \re(q) := (q+q^c)/2, \quad \im(q) := (q-q^c)/2,
  \qquad 
  q \in \alg.  
\end{equation}
Notice that in general $\re(q)$ and $\im(q)$ are not real numbers, at variance with the customary complex notations  $\Re(z) := (z + \overline{z})/2 \in \erre$ and $\Im(z) := (z - \overline{z})/2i \in \erre$ for $z \in \ci$. 

We finally observe that $qq^c \in \erre$ for every $q \in Q_\alg$ and we assume that
\begin{equation}\label{eq:assumption-2bis}
\begin{split}
 & \alg \text{ is endowed with a complete norm } |\cdot| \text{ such that }\\
 & \text{$|q_1q_2| \le |q_1||q_2|$ for every $q_1, q_2 \in \alg$ and }
 |q|^2=qq^c \text{ for every } q \in Q_\alg.
\end{split}
\end{equation}

The equivalence of the above definitions with other presentations (e.g. \cite{GhiPer11} is provided in \cite{GhiRec18}). We recall here that
 \begin{equation}\label{eq:intersection}
  \ci_\j \cap \ci_\k = \erre \qquad \forall \j, \k \in \mathbb{S}_\alg, \ \j \neq \pm \k.
\end{equation}

\begin{Example}[Clifford algebras]\label{examples of A}
For $n \in \enne \setmeno \{0\}$ let $\mc{P}(n)$ be the power set of $\{1,\ldots,n\}$. 
If we identify $\RR$ with the vector subspace $\RR \times \{0\}$ of $\RR^{2^n}=\RR \times \RR^{2^n-1}$ and we set $e_{\vuoto}:=1$, then we denote by $\{e_K\}_{K \in \mc{P}(n)}$ the canonical basis of $\RR^{2^n}$. For convenience, we set $e_k := e_{\{k\}}$ if $k \in \{1,\ldots,n\}$ and we define a real bilinear and associative product on $\erre^{2^n}$ by imposing that $1$ is the neutral element and that
\begin{align*} 
 & \text{$e_k^2=-1$ and $e_ke_h=-e_he_k$ if $k, h \in \{1,\ldots,n\}$ with $k \neq h$}, \\
 & \text{$e_K=e_{k_1}\cdots e_{k_s}$ if $K =\{k_1,\ldots,k_s\} \in \mc{P}(n) \setmeno \{\vuoto\}$  with 
 $k_1<\ldots<k_s$}.
\end{align*}
The \textit{Clifford conjugation} of $\erre^{2^n}$ is the $^*$-involution $q \longmapsto q^c := \overline{q}$  defined by
\[\textstyle
  \overline{q}:=\sum_{K \in \mc{P}(n)}(-1)^{|K|(|K|+1)/2}a_Ke_k \quad 
  \mbox{if } q=\sum_{K \in \mc{P}(n)}a_Ke_k \in \RR_n, \; a_K \in \RR,
\]
where $|K|$ indicates the cardinality of the set $K$.
Endowing $\erre^{2^n}$ with the above defined product and with the Clifford conjugation, we obtain a real $^*$-algebra $\alg$ satisfying \eqref{assumption on A}, called \emph{Clifford algebra $\mathit{C}\ell(0,n)$ of signature $(0,n)$}, which is denoted also by $\erre_{n}$. 
Observe that $\erre_1$ and $\erre_2$ are isomorphic to $\ci$ and $\quat$, respectively. Moreover $\erre_n$ is not commutative if $n \geq 2$. If $n \ge 3$ then $\erre_n$ has zero divisors, indeed 
$(1-e_{\{1,2,3\}})(1+e_{\{1,2,3\}})=0$. One verifies that a point $q=\sum_{K \in \mc{P}(n)}a_Ke_K$ of $\erre_n$ with $a_K \in \erre$ belongs to the quadratic cone $Q_{\erre_n}$ of $\erre_n$ if and only if it satisfies the following conditions 
\[
  a_K=0 \quad \mbox{and} \quad \langle q , qe_K \rangle_{2^n}=0 \quad 
\text{for every $K \in \mc{P}(n) \setmeno \{\vuoto\}$ with $e_K^2=1$},
\]
where $\langle \cdot , \cdot \rangle_{2^n}$ denotes the standard scalar product on $\RR^{2^n}$. 
On $\erre_n$ it is defined the following submultiplicative norm, called \textit{Clifford operator norm}:
$|q|_{\mathit{C}\ell}:=\sup\{|qa|_{2^n} \in \RR \, : \, |a|_{2^n}=1\}$,
where $|\cdot|_{2^n}$ indicates the Euclidean norm of $\erre^{2^n}$. It turns out that:
\begin{itemize}
 \item[(a)] $Q_{\erre_n}=\erre_n$ if and only if $n \in \{1,2\}$. In particular, $\erre_1$ and $\erre_2$ are division algebras.
 \item[(b)] $|q|_{\mathit{C}\ell}=|x|=\sqrt{x\overline{x}}$ for every $x \in Q_{\erre_n}$ and hence $|\cdot|_{\mathit{C}\ell}=|\cdot|$ if $n \in \{1,2\}$.
\end{itemize}
Notice that if $n \geq 3$ then the Euclidean norm $|\cdot|_{2^n}$ of $\erre_n$ is not submultiplicative (e.g. 
 $|(1+e_{\{1,2,3\}})^2|=\sqrt{8}>2=|1+e_{\{1,2,3\}}|^2$).
Endowing $\erre_n$ with Clifford conjugation and Clifford operator norm, we obtain a real $^*$-algebra 
$\alg$ satisfying \eqref{S_A nonvuota} and \eqref{eq:assumption-2bis}.
For further details we refer the reader to \cite{GM91,GHS08}.
\end{Example}

\begin{Example}[Complex numbers and quaternions]
If $n \in \enne \setmeno \{0\}$ and $\erre_n$ denotes the Clifford algebra of signature $(0,n)$ recalled in the previous Example \ref{examples of A}, then we have:
\begin{itemize}
\item[(i)] $\erre_1 = \ci$ with $e_1 = i$, where $z \longmapsto z^c=\bar{z}$ is the standard conjugation and $|\cdot|$ is the Euclidean norm;
\item[(ii)] $\erre_2$ is the algebra of quaternions $\quat$ with $i := e_1$, $j := e_2$, $k := e_3$, where $q = a+bi+cj+dk \longmapsto q^c=\bar{q} = a-bi-cj-dk$ and $|\cdot|$ is the euclidean norm.
\end{itemize}
\end{Example}

\begin{Def}
If $\alg$ satisfies \eqref{assumption on A} then a \emph{two-sided $\alg$-module} is a commutative group 
$(X,+)$ endowed with a left scalar multiplication $\alg \times X \funzione X : (q,x) \longmapsto qx$ and a right scalar multiplication $X \times \alg \funzione X : (x,q) \longmapsto xq$ such that 
\begin{alignat}{5}
 & q(x+y) = qx + qy, &\quad& (x+y)q = xq + yq	&\qquad&\forall x, y \in X,	&\quad &\forall q \in \alg, \notag \\
 & (p+q)x = px + qx, &\quad& x(p+q) = xp + xq	&\qquad& \forall x \in X,     &\quad &\forall p, q \in \alg, \notag \\   
 & 1x = x 	= x1	       &\quad&                              &\qquad& \forall x \in X,     &                                       \notag \\  
 & p(qx) = (pq)x,      &\quad& (xp)q = x(pq)  	&\qquad&    \forall x \in X, 	&\quad &\forall p, q \in \alg, \notag \\
 & p(xq) = (p x)q      &\quad& \qquad                 &\qquad&        \forall x \in X, &\quad &\forall p, q \in \alg,   \notag \\
 & r x = x r 	      &\quad&  \qquad                 &\qquad&   \forall x \in X, &\quad &\forall r \in \erre.  \notag
\end{alignat}
If $Y$ is a commutative subgroup of $X$ then $Y$ is called a \emph{left $\alg$-submodule} if $qx \in Y$ whenever $x \in Y$ and $q \in \alg$. Instead $Y$ is called a \emph{right $\alg$-submodule} of $X$ if $xq \in Y$ whenever $x \in Y$ and $q \in \alg$. Finally $Y$ is called a \emph{two-sided $\alg$-submodule} of $X$ if it is both a left and a right $\alg$-submodule of $X$.
\end{Def}

\begin{Def}
Assume \eqref{assumption on A} and \eqref{eq:assumption-2bis} hold. A two-sided $\alg$-module $X$ is called a \emph{normed two-sided $\alg$-module} if it is endowed with a \emph{$\alg$-norm on $X$}, that is, a function
$\norma{\cdot}{} : X  \funzione \clsxint{0,\infty}$ such that
\begin{alignat}{3}
  & \norma{x}{} = 0 \ \Longleftrightarrow \ x = 0, \notag \\
  & \norma{x + y}{} \le \norma{x}{} + \norma{y}{} 		& \qquad & \forall x, y \in X,  \notag \\
  & \norma{q x}{} \le |q| \, \norma{x}{}, \quad  \norma{x q}{} \le |q| \, \norma{x}{} 	
  & \qquad & \forall x \in X, \quad \forall q \in \alg. \label{eq:homog} 
\end{alignat}
We say that $X$ is a \emph{Banach two-sided $\alg$-module} if the metric 
$d : X \times X \funzione \clsxint{0,\infty} : (x,y) \longmapsto \norma{x-y}{}$ is complete.
\end{Def}

Let us recall the following result (cf. \cite[Lemma 3.3]{GhiRec18}).

\begin{Lem}\label{L:norma moltiplicativa}
Assume \eqref{assumption on A} and \eqref{eq:assumption-2bis} hold, and let $X$ be a normed two-sided $\alg$-module. Then
\begin{equation}\label{norma moltiplicativa}
\norma{qx}{} = \norma{xq}{} = |q|\norma{x}{} \qquad \forall x \in X, \quad \forall q \in Q_\alg.
\end{equation}
\end{Lem}

\begin{Def}
Assume \eqref{assumption on A} holds and that $X$ is a two-sided $\alg$-module. Let $D(\A)$ be a right $\alg$-submodule of $X$. We say that $\A : D(\A) \funzione X$ is \emph{right linear} if it is additive and
\begin{equation*}
 \A(xq) = \A(x) q \qquad \forall x \in D(\A), \quad \forall q \in \alg.
\end{equation*}
As usual, the notation $\A x$ is often used in place of $\A(x)$. We use the symbol $\End^{\r}(X)$ to denote the set of right linear operators $\A$ with $D(\A) = X$. The identity operator is right linear and is denoted by $\Id_X$ or simply by $\Id$. Moreover, if $X$ is a normed two-sided $\alg$-module, then we say that 
$\A : D(\A) \funzione X$ is \emph{closed} if its graph is closed in $X \times X$. As in the classical theory, we set 
$D(\A^n) := \{x \in D(\A^{n-1}) \, : \, \A^{n-1} x \in D(\A)\}$ for every $n \in \enne \setmeno \{0\}$, where $\A^0:=\Id$. 
\end{Def}

Let us also recall the following definition (see, e.g., \cite[Chapter 1, p. 55-57]{AndFul74}).

\begin{Def}
Let $D(\A)$ be a right $\alg$-submodule of $X$ and let $q \in \alg$. If $\A : D(\A) \funzione X$ is a right linear operator, then we define the mapping $q \A : D(\A) \funzione X$ by setting
\begin{equation}
   (q \A)(x) := q \A(x), \qquad 
   x \in D(\A).  \label{operatore per scalare a sx}
\end{equation}  
If $D(\A)$ is also a left $\alg$-submodule of $X$, then we can define $\A q :D(\A) \funzione X$ by setting
\begin{equation}
(\A q)(x):=\A(q x), \qquad 
x \in D(\A). \label{operatore per scalare a dx}
\end{equation} 
The sum of operators is defined in the usual way.
\end{Def}
It is easy to see that the operators defined in \eqref{operatore per scalare a sx} and \eqref{operatore per scalare a dx} are right linear.

\begin{Def}
Assume $X$ is normed with $\alg$-norm $\norma{\cdot}{}$. For every $\B \in \End^\r(X)$, we set
\begin{equation}\label{norma operatoriale}
  \norma{\B}{} := \sup_{x \neq 0} \frac{\norma{\B x}{}}{\norma{x}{}}
\end{equation}
and we define the set
$\Lin^\r(X) := \{\B \in \End^\r(X)\ :\ \norma{\B}{} < \infty\}.$
\end{Def}

\section{Main results in their general form}

In order to state our main result in its general form we recall the noncommutative spectral notions given for the first time in \cite{CoGeSaSt07} for quaternions and in \cite{ColSabStr08} for arbitrary Clifford algebras $\erre_n$.
Here we consider the general case introduced in \cite[Definition 2.26]{GhiRec16}.
We will assume that
\[
\text{\emph{$\alg$ satisfies \eqref{assumption on A}, \eqref{S_A nonvuota} and \eqref{eq:assumption-2bis}, and $X$ is a Banach two-sided $\alg$-module.}}
\]

\begin{Def}\label{D:spherical spectral notions}
Let $D(\A)$ be a right $\alg$-submodule of $X$ and let $\A : D(\A) \funzione X$ be a closed right linear operator. \begin{itemize}
\item[(i)]
Given $q \in Q_\alg$, the right linear operator $\Delta_q(\A) : D(\A^2) \funzione X$ is defined by
\begin{equation*}
  \Delta_q(\A) := \A^2 - 2\re (q) \,\A + |q|^2 \Id, \qquad q \in Q_\alg.
\end{equation*}
\item[(ii)]
The \emph{spherical resolvent set $\rho_\s(\A)$ of $\A$} is defined by
\begin{equation*}
\rho_\s(\A) := \{q \in Q_\alg \, : \, \text{$\Delta_q(\A)$ is bijective, $\Delta_q(\A)^{-1} \in \Lin^\r(X)$}\}
\end{equation*}
and the \emph{spherical spectrum $\sigma_\s(\A)$ of $\A$} by $\sigma_\s(\A) :=Q_\alg \setmeno \rho_\s(\A)$.
\item[(iii)]
Given $q \in \rho_\s(\A)$, the \emph{spherical quasi-resolvent operator} of $\A$ at $q$ 
is the operator $\Q_q(\A) : X \funzione X$ defined by
\begin{equation*}
  \Q_q(\A) := \Delta_q(\A)^{-1}, \qquad q \in \rho_\s(\A).
\end{equation*}
\item[(iv)]
Given $q \in \rho_\s(\A)$, the \emph{spherical resolvent of $\A$ at $q$} is the operator $\C_q(\A) : X \funzione X$ defined by
\begin{equation}\label{resolvent operator}
\C_q(\A) := \Q_q(\A) q^c - \A \Q_q(\A), \qquad q \in \rho_\s(\A). \notag
\end{equation}
\end{itemize}
\end{Def}

Let us observe that 
\begin{equation}\label{resolvent bounded}
  \Q_q(\A) \in \Lin^\r(X), \qquad \C_q(\A) \in \Lin^\r(X) \qquad \forall q \in \rho_\s(\A).
\end{equation}
Indeed, by definition, $\Q_q(\A)$ is bounded and if we endow $(X,+)$ with the (left) real scalar multiplication $\erre \times X \funzione X : (r,x) \longmapsto rx = xr$, then thanks to \eqref{norma moltiplicativa} $X$ can be considered as a real Banach space and $\A$ is a closed $\erre$-linear operator on it, thus the closed graph theorem implies that 
$\A\Q_q(\A)$ is continuous and consequently $\C_q(\A)$ is also continuous. Since all these operators are also 
$\alg$-right linear we infer \eqref{resolvent bounded}.

We mention that a definition that has some similarities with the spherical spectrum was given in \cite{Kap} in the context of real $^*$-algebras.

We now recall the natural definition of right linear operator semigroup (cf. \cite{ColSab11} for the quaternionic case and  \cite{GhiRec16} for the general case).

\begin{Def}
A mapping $\S : \clsxint{0,\infty} \funzione \Lin^\r(\X)$ is called \emph{strongly continuous} if the 
$t \longmapsto \S(t)x$ is continuous from $\clsxint{0,\infty}$ into $X$ for every $x \in X$.
\end{Def}

\begin{Def}\label{def semigroup}
A mapping $\T : \clsxint{0,\infty} \funzione \Lin^\r(X)$ is called \emph{right linear strongly conti\-nuous (operator) semigroup} if $\T$ is strongly continuous and if 
\begin{align*}
  & \T(t+s) = \T(t)\T(s) \qquad \forall t, s > 0, \\
  & \T(0) = \Id.
\end{align*}
The \emph{generator of $\T$} is the right linear operator $\A : D(\A) \funzione X$ defined by
\begin{align*}
  & D(\A) := \left\{x \in X\ :\ \exists \lim_{h \to 0} \frac{1}{h}(\T(h)x - x) \in X\right\}, \\
  & \A x := \lim_{h \to 0}\frac{1}{h}(\T(h)x - x), \qquad x \in D(\A).
\end{align*}
\end{Def}

For the classical theory of semigroups in the complex framework we refer, e.g.,  to 
\cite{HilPhi57,Dav80,Paz83,Gol85,Lun95,Tai95,EngNag00}.

The next result includes Theorem \ref{thm:supermain-clifford}. 

\begin{Thm}\label{thm:supermain}
Let $m\in\enne$, let $P(\x)=\sum_{k=0}^{m+2}\x^ka_k\in\RR[\x]$ such that $a_{m+2}\neq0$ and let $g_P\in C^\infty(\clsxint{0,\infty};\RR)$ be the unique solution of \eqref{ODE}. Let $\T:\clsxint{0,\infty} \funzione \Lin^r(X)$ be a strongly continuous right linear semigroup, let 
$\A:D(\A) \funzione X$ be its generator and let $\omega\in\erre$ be a real constant such that $M:=\sup_{t\in\clsxint{0,\infty}}\norma{\T(t)}{}e^{-\omega t}<\infty$. Then, if $r_P>\omega$, the operator $P(\A) = \sum_{k=0}^{m+2}\A^ka_k$ is bijective, $P(\A)^{-1}\in\Lin^r(X)$ and it holds:
\begin{itemize}
 \item[$(\mr{a})$] $P(\A)^{-1}=\Lap(g_P)$, that is,
 \begin{equation}\label{inverso P = integrale}
 P(\A)^{-1}x=\int_0^\infty\T(t)g_P(t)x \de t \qquad \forall x\in X.
 \end{equation}
 \item[$(\mr{b})$] Given $(p_0,p_1,\ldots,p_{m+1})\in\alg^{m+2}$, we have
 \begin{equation}\label{potA P-1}
 \sum_{j=0}^{m+1}\A^kP(\A)^{-1}p_jx=\Lap\left(\sum_{j=0}^{m+1}g_P^{(j)}(-1)^jp_j\right)x \qquad \forall x\in X.
 \end{equation}
\end{itemize}
Moreover,
 \begin{equation}\label{estimate for P-1}
 \norma{P(\A)^{-1}}{}  \le M\sum_{j=1}^h\sum_{k=1}^{m_j}\frac{|c_{-k}^{(j)}|}{(r_P-\omega)^k}
 \end{equation}
 where $c_{-k}^{(j)}$ is the residue at $\lambda_j$ of the complex rational function $a_{m+2}(z-\lambda_j)^{k-1}/P(-z)$.
\end{Thm}

Furthermore we have

\begin{Thm}\label{thm:main}
The statement of Theorem \ref{thm:main-clifford} holds true replacing $\RR_n$ with $\alg$.
\end{Thm}


\section{Proofs}\label{S:proofs}

Let us start with a lemma on strongly continuous mapping. The symbol $C(\clsxint{0,\infty};\alg)$ denotes the space of continuous functions from $\clsxint{0,\infty}$ to $\alg$, both endowed with the topo\-logy induced by the Euclidean distance.

\begin{Lem}\label{T(.)g(.)x continuous}
If $\T : \clsxint{0,\infty} \funzione \Lin^r(X)$ is a strongly continuous and $g \in C(\clsxint{0,\infty};\alg)$, then the following statements hold true.
\begin{itemize}
\item[(a)]
The function $t \longmapsto \norma{\T(t)}{}|g(t)|$ is Lebesgue measurable on $\clsxint{0,\infty}$.
\item[(b)]
For every $x \in X$ the function $t \longmapsto \T(t)g(t)x$ is continuous from $\clsxint{0,\infty}$ into $X$.
\end{itemize}
\end{Lem}

\begin{proof}
Since $\T:\clsxint{0,\infty} \funzione \Lin^r(X)$ is strongly continuous, by the Banach-Steinhaus theorem it follows that $\norma{\T(t)}{} \le \liminf_{\tau \to t} \norma{\T(\tau)}{}$ for every $t \ge 0$, i.e.
$t \longmapsto \norma{\T(t)}{}$ is lower semicontinuous, and hence it is Lebesgue measurable.
Thus (a) is proved. In order to prove (b) fix an arbitrary $t_0 \ge 0$. Since $\T$ is strongly continuous, by the uniform boundedness principle there exists $C > 0$ such that for every $t \ge 0$ with $|t-t_0| < 1$ we have $\norma{T(t)}{} \le C$ and 
\begin{align*}
  \norma{\T(t)g(t)x - \T(t_0)g(t_0)x}{} 
    & \le \norma{\T(t)g(t)x - \T(t)g(t_0)x}{} + \norma{\T(t)g(t_0)x - \T(t_0)g(t_0)x}{}  \\
    & \le C|g(t) - g(t_0)|\norma{x}{} + \norma{\T(t)g(t_0)x - \T(t_0)g(t_0)x}{}.
\end{align*}
Thus the continuity of $t \longmapsto \T(t)g(t)x$ at $t_0$ follows from the continuity of $g$ and from the strong continuity of $\T$.
\end{proof}

If $\T : \clsxint{0,\infty} \funzione \Lin^r(X)$ is a strongly continuous right linear semigroup, $t \ge 0$,
$g$ $\in$ $C(\clsxint{0,\infty};\alg)$, and $x \in X$, then Lemma \ref{T(.)g(.)x continuous} and estimate 
$\norma{\T(t)g(t)x}{} \le \norma{\T(t)}{}|g(t)|\norma{x}{}$ allow to give the following definition.

\begin{Def}
Let $\T:\clsxint{0,\infty} \funzione \Lin^r(X)$ be a strongly continuous right linear semigroup. We denote by 
$L_\T(\clsxint{0,\infty};\alg)$ the real vector space of all continuous functions $g : \clsxint{0,\infty} \funzione \alg$ such that the function $t \longmapsto \norma{\T(t)}{}|g(t)|$ belongs to $L^1(\clsxint{0,\infty};\erre)$. For every 
$g \in L_\T(\clsxint{0,\infty};\alg)$ we define the operator $\Lap(g) : X \funzione X$ by setting
\begin{equation}\label{def Lap}
  \Lap(g)x:=\int_0^\infty\T(t)g(t)x\de t, \qquad x \in X.
\end{equation}
\end{Def}

Notice that the assumptions implies that the integral in \eqref{def Lap} is a convergent Lebesgue integral
for functions with values in the Banach space $(X,+)$ endowed with the real scalar multiplication $\erre \times X \funzione X : (r,x) \longmapsto rx = xr$ (thanks to \eqref{norma moltiplicativa} $\norma{\cdot}{}$
is a norm on this real vector space). The symbol $L^1(J;X)$ denotes the space of Lebesgue integrable functions from an interval 
$J \subseteq \erre$ into this real Banach space.

In the remainder of the paper, for $g \in C(\clsxint{0,\infty};\alg)$, the symbols $g'$ and $g''$ will denote
the first and second derivative of $g$, respectively.

\begin{Lem}
If $\T:\clsxint{0,\infty} \funzione \Lin^r(X)$ is a strongly continuous right linear semigroup, then the following statements hold true.
\begin{itemize}
\item[(a)]
If $g \in L_\T(\clsxint{0,\infty};\alg)$ then $t \longmapsto \T(t)g(t-h)x$ belongs to $L^1(\clsxint{h,\infty};X)$ for every $h > 0$ and for every $x \in X$.
\item[(b)]
If $g, g' \in L_\T(\clsxint{0,\infty};\alg)$ then 
\[
  \lim_{h \to 0} \left\|\int_h^\infty \T(t)\frac{g(t-h) - g(t)}{h}x \de t + \Lap(g')x \right\| = 0
  \qquad \forall x \in X.
\]
\end{itemize}
\end{Lem}

\begin{proof}
The claim (a) follows trivially by the estimate 
$\norma{\T(t)g(t-h)x}{} \le \norma{\T(h)}{}\norma{\T(t-h)}{}|g(t-h)|\norma{x}{}$ holding for every $x \in X$, $h > 0$, and $t > h$.
In order to prove (b) fix $x \in X$ and an arbitrary $\eps > 0$, and let $T > 0$ be such that 
$\int_T^\infty \norma{\T(t)}{}|g(t)| \de t < \eps$. Since $\T$ is a strongly continuous semigroup, there exists 
$M \ge 1$ such that $\norma{\T(s)}{} \le M$ whenever $0 \le s \le 1$, therefore for every $h \in \opint{0,1}$ and every $t > h$ we have
\begin{align}
  \left\|\T(t)\frac{g(t-h) - g(t)}{h}x\right\| 
    & = \left\|\int_0^1\T(\xi h)\T(t-\xi h)g'(t-\xi h) x \de \xi \right\|  \notag \\
    & \le M \norma{x}{} \int_0^1\norma{\T(t-\xi h)}{}|g'(t-\xi h)|\de \xi  \notag 
\end{align}
hence it follows that
\begin{align}\label{estimate per le code}
   \left\| \int_T^\infty \T(t)\frac{g(t-h) - g(t)}{h} x \de t  \right\|  
   & \le M \norma{x}{} \int_T^\infty \int_0^1  \norma{\T(t-\xi h)}{}|g'(t-\xi h)|\de \xi \de t\notag \\
   & =   M \norma{x}{} \int_0^1 \int_T^\infty  \norma{\T(t-\xi h)}{}|g'(t-\xi h)|\de t \de \xi \notag \\
   & =   M \norma{x}{} \int_0^1 \int_{T-\xi h}^\infty  \norma{\T(\tau)}{}|g'(\tau)|\de \tau \de \xi \notag \\
   & \le M \norma{x}{} \int_0^1 \int_{T}^\infty  \norma{\T(\tau)}{}|g'(\tau)|\de \tau \de \xi \notag \\
   & \le M \norma{x}{} \int_{T}^\infty  \norma{\T(\tau)}{}|g'(\tau)|\de \tau  \le M\norma{x}{}\eps.
\end{align}
Moreover it is easily found a $\delta \in \opint{0,1}$ such that for every $h \in \opint{0,\delta}$ we have 
\begin{align}\label{estimate al finito}
   & \left\| \int_h^T \T(t)\frac{g(t-h) - g(t)}{h} x \de t + \int_0^T \T(t)g'(t) x \de t \right\|  \notag \\
   & \le \left\| \int_h^T \T(t)\left(\frac{g(t-h) - g(t)}{h} + g'(t)\right) x \de t \right\| +
    \left\| \int_0^h \T(t)g'(t) x \de t \right\| \le \eps. 
\end{align}
Hence assertion (b) follows from \eqref{estimate per le code}--\eqref{estimate al finito} and from the following estimate
\begin{align}
  & \left\| \int_h^\infty \T(t)\frac{g(t-h) - g(t)}{h} x \de t + \Lap(g')x  \de t \right\| \notag \\
  & \le \left\| \int_h^T \T(t)\frac{g(t-h) - g(t)}{h} x \de t + \int_0^T \T(t)g'(t) x \de t \right\| \notag \\
  & \sp + \left\| \int_T^\infty \T(t)\frac{g(t-h) - g(t)}{h} x \de t\right\| + \left\|\int_T^\infty \T(t)g'(t) x \de t \right\|. \notag
\end{align}
\end{proof}

The next lemma plays a key role in the proof of Theorem \ref{thm:main}.

\begin{Lem}\label{lem:L}
Let  $\T:\clsxint{0,\infty} \funzione \Lin^r(X)$ be a strongly continuous right linear semigroup and for every 
$g \in L_\T(\clsxint{0,\infty};\alg)$ let $\Lap(g) : X \funzione X$ be defined by \eqref{def Lap}.
Then $\Lap(g) \in \Lin^\r(X)$ for every $g \in L_\T(\clsxint{0,\infty};\alg)$ and the resulting mapping 
$\Lap :  L_\T(\clsxint{0,\infty};\alg) \funzione \Lin^\r(X)$ is $\erre$-linear. Moreover the following assertions hold.
\begin{itemize}
\item[(a)] 
If $g, g'  \in L_\T(\clsxint{0,\infty};\alg)$, then 
  \begin{align}
    &\Lap(g)(X) \subseteq D(\A), \label{eq:LA}\\
    &\label{eq:AL}
    \A\Lap(g)x=-g(0)x-\Lap(g')x \qquad \forall x \in X.
  \end{align}
  
\item[(b)]  
If $m \in \enne$, $g, g', \ldots, g^{(m+2)} \in L_\T(\clsxint{0,\infty};\alg)$, and 
$g(0) = g'(0) = \cdots = g^{(m)}(0) = 0$, then 
  \begin{align}
    &\Lap(g)(X) \subseteq D(\A^{m+2}), \label{eq:LAm}\\
    & \A^{m+2}\Lap(g)x=(-1)^{m+2}\left(g^{(m+1)}(0)x + \Lap(g^{(m+2)})x\right) \qquad \forall x \in X,\label{eq:ALm}\\
    & \A^{k}\Lap(g)x = (-1)^k\Lap(g^{(k)}) \qquad \forall x \in X,\ \forall k \in \{0,\ldots,m+1\}. \label{ALk} 
  \end{align}

\item[(c)] 
If $g, g', g'' \in L_\T(\clsxint{0,\infty};\alg)$, $g(0)=0$, and $g'(0)=1$, then
\begin{align}
&\label{eq:DeltaL}
\Delta_q(\A)\Lap(g)x=x+\Lap(g''+2\re(q)g'+|q|^2g)x \qquad \forall x \in X, \ \forall q \in Q_\alg.
\end{align}

\item[(d)] 
If $g, g' \in L_\T(\clsxint{0,\infty};\alg)$ and $g$ is real-valued, then
\begin{equation}
\label{eq:commutative1}
\A\Lap(g)x=\Lap(g)\A x \qquad \forall x \in D(\A).
\end{equation}

\item[(e)]
If $m \in \enne$, $g, g', \ldots, g^{(m+2)} \in L_\T(\clsxint{0,\infty};\alg)$, 
$g(0) = g'(0) = \cdots = g^{(m)}(0) = 0$, and $g$ is real valued, then 
\begin{equation}
\label{eq:commutative2}
\A^k\Lap(g)x=\Lap(g)\A^kx \qquad \forall x \in D(\A^k),\ \forall k \in \{1,\ldots,m+2\}.
\end{equation}
\end{itemize}
\end{Lem}

\begin{proof}
For every $g \in L_\T(\clsxint{0,\infty};\alg)$ the right linearity of $\Lap(g)$ follows from the right linearity of $\T(t)$ and from the definition of the $X$-valued Lebesgue integral. For every $x \in X$ we have
\[
  \norma{\Lap(g)x}{} \le  \norma{x}{}\int_0^\infty \norma{\T(t)}{}|g(t)|\de t,
\]
hence $\Lap(g)$ is also continuous and $\norma{\Lap(g)}{} \le  \int_0^\infty \norma{\T(t)}{}|g(t)|\de t$. The real linearity of $\Lap$ is straightforward. Now in the following list of items we prove the assertions from (a) to (e).

\noindent
(a)
For every $h > 0$ and for every $x \in X$ we have
\begin{align*}
  \frac{\T(h)-\Id}{h}\,\Lap(g)x 
    & = \frac{1}{h}\int_0^\infty\T(t+h)g(t)x\de t-\frac{1}{h}\int_0^\infty\T(t)g(t)x\de t \\
    & = \frac{1}{h}\int_h^\infty\T(t)g(t-h)x\de t-\frac{1}{h}\int_0^\infty\T(t)g(t)x\de t \\
    & = -\frac{1}{h}\int_0^h\T(t)g(t)x\de t+\int_h^\infty\T(t)\left(\frac{g(t-h)-g(t)}{h}\right)x\de t.
\end{align*}
Hence, taking the limit as $h \to 0$, thanks to Lemma \ref{T(.)g(.)x continuous} we find that 
\[
  \A\Lap(g)x = -\T(0)g(0)x - \int_0^\infty \T(t)g'(t)x \de t = - g(0)x - \Lap(g')x.
\]

\noindent
(b)
We proceed by induction on $m\in\{-1\}\cup\enne$. The case $m=-1$ follows from (a). Let us assume that the result is true for $m-1$, and we prove it for $m$. Therefore if $g$ satisfies the assumptions, in particular we have
$\Lap(g)x \in D(\A^{m+1})$ and $\A^{m+1}\Lap(g)x = (-1)^{m+1}(g^{(m)}(0)x + \Lap(g^{(m+1)})x$ for every $x \in X$. But $g^{(m)}(0) = 0$ hence $\A^{m+1}\Lap(g)x = (-1)^{m+1}\Lap(g^{(m+1)})x$ and 
$\Lap(g^{(m+1)})x \in D(\A)$ by virtue of an application of \eqref{eq:LA} with $g$ replaced by $g^{(m+1)}$. Thus
$\A^{m+1}\Lap(g)x \in D(\A)$ and \eqref{eq:LAm} follows. Using again the validity of the statement for $m-1$ and the identity $g^{(m)}(0) = 0$ we have
\begin{align}
  \A^{m+2}\Lap(g)x 
   & = \A\A^{m+1}\Lap(g)x 
     = (-1)^{m+1}\A\Lap(g^{(m+1)})x  \notag \\
   & = (-1)^m(g^{(m+1)}(0)x +\Lap(g^{(m+2)}))x, \notag
\end{align}
where in the last equality we have used \eqref{eq:AL} with $g$ replaced by $g^{(m+1)}$. Therefore 
\eqref{eq:ALm} is proved. Formula \eqref{ALk} is trivial for $k=0$ and follows from (a) for $k=1$, while 
for $2 \le k \le m+1$ follows from \eqref{eq:ALm} which we have already proved.

\noindent
(c)
Now fix $x \in X$ and $q \in Q_\alg$. From (b) we obtain $\A^2\Lap(g)x =x+\Lap(g'')x$, 
hence, exploiting again \eqref{eq:AL} and the $\erre$-linearity of $\Lap$, we obtain
\begin{align*}
\Delta_q(\A)\Lap(g)x 
  & = \A^2\Lap(g)x - 2\re(q)\A\Lap(g)x + |q|^2\Lap(g)x \\
  &=x+\Lap(g'')x-2\re(q)(-\Lap(g')x)+|q|^2\Lap(g)x\\
  &=x+\Lap(g''+2\re(q)g'+|q|^2g)x.
\end{align*}

\noindent
(d) If $x \in D(\A)$, then for every $h > 0$ and for every $t > 0$ we have 
$\T(t)g(t)\T(h)x$ $=$ $\T(t)\T(h)g(t)x$ $=$ $\T(t+h)g(t)x$ $=$ $\T(h)\T(t)g(t)x$, because $g$ is real-valued and 
$\T$ is a semigroup. Therefore
\begin{align*}
   \Lap(g)\frac{\T(h)-\Id}{h} x 
    & = \frac{1}{h}\int_0^\infty\T(t)g(t)\T(h)x\de t - \frac{1}{h}\int_0^\infty\T(t)g(t)x\de t \notag \\
    & = \frac{1}{h}\int_0^\infty\T(h)\T(t)g(t)x\de t - \frac{1}{h}\int_0^\infty\T(t)g(t)x\de t \notag \\
    & = \frac{\T(h)-\Id}{h}\Lap(g)x,
\end{align*}
and the assertion follows taking the limit as $h \to 0$ and invoking \eqref{eq:LA}.

\noindent
(e) By induction on $m\in\{-1\}\cup\enne$. The case $m=-1$ is true by virtue of (d). Let us assume that the result is true for $m-1$, and we prove it for $m$. Therefore if $g$ satisfies the assumptions, in particular we have that
$\A^{k}\Lap(g) = \Lap(g)\A^{k}$ on $D(\A^{k})$ for all $k \in \{1,\ldots,m+1\}$. Hence if $x \in D(\A^{m+2})\subset D(\A^{m+1})$ then $\A^{m+1}x \in D(\A)$ and 
we have that
\[
  \Lap(g)\A^{m+2} x = \Lap(g)\A \A^{m+1}x =\A\Lap(g)\A^{m+1}x = \A\A^{m+1}\Lap(g)x = \A^{m+2}\Lap(g)x,
\]
where in the second equality we have used again (d).
\end{proof}

\begin{proof}[Proof of Theorem \ref{thm:supermain}]
Let us first recall that the existence of constant $\omega \in \erre$ such 
that $M := \sup_{t \ge 0}\norma{\T(t)}{}e^{-\omega t} < \infty $ is well know (cf. \cite[Thm 4.5(b)]{GhiRec16}). 
From \eqref{g_P} it follows that $g_P \in C^{m+2}(\clsxint{0,\infty};\alg)$ and the Leibniz formula for the 
higher derivatives of a product yields the existence of a polynomial $p(t, \lambda_1, \ldots, \lambda_h)$ such that
\[
  \norma{\T(t)}{}|g_P^{k}(t)| \le Me^{\omega t}|p(t, \lambda_1, \ldots, \lambda_h)| e^{-tr_P} 
  = M|p(t, \lambda_1, \ldots, \lambda_h)|e^{(\omega-r_P)t}
\]
for all $t\geq0$ and for all $k\in\{0,1,\ldots,m+2\}$. Thus $g_P, g_P', \ldots, g_P^{(m+2)} \in L_\T(\clsxint{0,\infty};\alg)$ and we can apply part (b) of Lemma \ref{lem:L} taking into account of initial conditions for $g_P$. We infer that 
\begin{align}
  P(\A)\Lap(g_P) x 
    & =  \sum_{k=0}^{m+2} \A^k a_k\Lap(g_P)x \notag \\
    & =  \sum_{k=0}^{m+1} (-1)^k\Lap(g_P^{(k)})a_k x +
         (-1)^{m+2}\left(g_P^{(m+1)}(0) + \Lap(g_P^{(m+2)})\right)a_{m+2} x\notag \\
     & =  \Lap\left(\sum_{k=0}^{m+2} g_P^{(k)}(-1)^ka_k\right)x + (-1)^{m+1}x = xg_P^{(m+1)}(0)(-1)^ma_{m+2}=x. \notag
\end{align}
On the other hand since $g_P$ is real-valued we can also apply parts (d) and (e) of Lemma \ref{lem:L} and infer that $\Lap(g_P)P(\A) = P(\A)\Lap(g_P)$ thus $\Lap(g_P) = P(\A)^{-1}$ and \eqref{inverso P = integrale} is proved. Now take $(p_0,p_1,\ldots,p_{m+1})\in\alg^{m+2}$. Using Lemma \ref{lem:L}(b) we deduce:
\begin{align}
  \sum_{j=0}^{m+1}\A^kP(\A)^{-1}p_jx 
   & = \sum_{j=0}^{m+1}\A^k\Lap(g_P)p_jx = \sum_{j=0}^{m+1} (-1)^k\Lap(g_P^{(k)})p_jx  = \Lap\left(\sum_{j=0}^{m+1} (-1)^kg_P^{(k)}p_j\right)x \notag\
\end{align}
and \eqref{potA P-1} is proved. It is well known that
\[
  Q_j(t) = \sum_{k=1}^{m_j}\frac{c_{-k}^{(j)}}{(k-1)!}t^{k-1}\,,
\] 
where $c_{-k}^{(j)}$ is the coefficient of $(z-\lambda_j)^{-k}$ within the partial frations decomposition of 
$a_{m+2}/P(-z)$, i.e. the residue at $\lambda_j$ of the function $a_{m+2}(z-\lambda_j)^{k-1}/P(-z)$. Therefore 
\[
  |g_P(t)| \le \sum_{j=1}^h\sum_{k=1}^{m_j}\frac{|c_{-k}^{(j)}|}{(k-1)!}t^{k-1}e^{-r_Pt} \qquad \forall t \ge 0
\]
and
\begin{align}
\norma{P(\A)^{-1}}{} 
  & \le \int_0^\infty \norma{\T(t)}{}\sum_{j=1}^h\sum_{k=1}^{m_j}\frac{|c_{-k}^{(j)}|}{(k-1)!}t^{k-1}e^{-r_Pt}\de t \notag \\
  & \le M\sum_{j=1}^h\sum_{k=1}^{m_j}\frac{|c_{-k}^{(j)}|}{(k-1)!} \int_0^\infty t^{k-1}e^{-(r_P-\omega)t} \de t
  \notag\\
  & = M\sum_{j=1}^h\sum_{k=1}^{m_j}\frac{|c_{-k}^{(j)}|}{(r_P-\omega)^k} \frac{1}{(k-1)!}\int_0^\infty t^{k-1}e^{-t} \de t
  \notag\\
  & = M\sum_{j=1}^h\sum_{k=1}^{m_j}\frac{|c_{-k}^{(j)}|}{(r_P-\omega)^k} \notag
\end{align}
and the theorem is completely proved.
\end{proof}

Now we address the case when $P(\x) = \x^2 -2\re(q) + |q|^2$ for some $q \in Q_\alg$ which is related to part (c)
of Lemma \ref{lem:L}. We first present a simple lemma whose proof is a trivial calculus exercise.

\begin{Lem} \label{lem:C}
For every fixed $q \in Q_\alg$, the unique solution of the Cauchy problem
\begin{equation}\label{Cauchy pb}
g''+2\re(q)g'+|q|^2g=0,  \quad  g(0)=0, \quad g'(0)=1.
\end{equation}
is the function $g_q: \erre \funzione \erre$ defined by 
\begin{equation}\label{g_q}
g_q(t) := te^{-\re(q)t}\sinc(t|\im(q)|), \qquad t \in \erre,
\end{equation} 
where we recall that $\sinc : \erre \funzione \erre$ is the \emph{unnormalized sinc function}, i.e. the only continuous real function $\xi$ on $\erre$ such that $\xi(r) = (\sin r)/r$ for all $r \neq 0$.
Moreover we have that $g_q, g_q', g_q'' \in L_\T(\clsxint{0,\infty};\alg)$.
\end{Lem}

Now we are in position to find the integral representations of the spherical quasi-resolvent operator and of the spherical resolvent operators as a simple consequence of our main Theorem \ref{thm:supermain}.

\begin{proof}[Proof of Theorem \ref{thm:main}]
In order to prove \eqref{eq:Q_S} it is enough to apply part (a) of Theorem \ref{thm:supermain} with 
$P(\x) = \x^2 -2\re(q) + |q|^2$ taking into account of Lemma \ref{lem:C}. 
Formula \eqref{eq:AQ_S} is a straighforward application of part (b) of Theorem \ref{thm:supermain} with
$p_0=0$, $p_1=1$ and $p_2 = 0$. Instead \eqref{eq:C_S} is obtained taking  in part (b) of Theorem \ref{thm:supermain}  $p_0 = q^c$, $p_1 = -1$, and $p_2 = 0$.
Finally 
\begin{align}
  \norma{\Q_q(\A)}{} 
   \le \int_0^\infty \norma{\T(t)}{}|g_p(t)|\de t  \le  \int_0^\infty Mte^{(\omega - \re(q))t} \de t 
    = \frac{M}{(\re(q) - \omega)^2},\notag
\end{align}
i.e. \eqref{eq:estimateQ} holds. Finally estimate \eqref{eq:estimateC} can be obtained in a similar way.
\end{proof}

\begin{Rem}
By exploiting \eqref{g_q} of our Lemma \ref{lem:C} we can write the integral representation
\eqref{eq:Q_S} in a more explicit way for every $q \in Q_\alg$ such that $\re(q) > \omega$:
\begin{equation}
\Q_q(\A)x=-\int_0^\infty\T(t)\frac{e^{-\re(q)t}\sin\left(|\im(q)| t\right)}{|\im(q)|} x\de t  
\qquad \forall x \in X, \ q \not\in \erre,
\end{equation}
\begin{equation}
\Q_q(\A)x=-\int_0^\infty\T(t)te^{-qt} x\de t  
\qquad \forall x \in X, \ q \in \erre.
\end{equation}
\end{Rem}

The following lemma will connect the integral representation of $\Q_q(\A)$ to the so called \emph{spherical
derivative} of $q \longmapsto e^{-tq}$ (cf. \cite{GhiPer11}).

\begin{Lem}\label{L:exp^t}
For every $t \in \erre$ let $\exp^t : Q_\alg \funzione \alg$ be the function defined by 
\begin{equation}
  \exp^t(q) := \sum_{n=0}^\infty \frac{t^n}{n!}q^n = \sum_{n=0}^\infty \frac{(tq)^n}{n!}, \qquad q \in Q_\alg.
\end{equation}
If $(\exp^t)'_s : Q_\alg \setmeno \erre \funzione \alg$ denotes the function defined by
\begin{equation}
  (\exp^t)'_s(q) := (q-q^c)^{-1}(\exp^t(q) - \exp^t(q^c)), \qquad q \in Q_\alg \setmeno \erre,
\end{equation}
which is also called \emph{spherical derivative} of $\exp^t$, then
$(\exp^t)'_s$ extends to a unique conti\-nuous function on $Q_\alg$, which we still denote by $(\exp^t)'_s : Q_\alg \funzione \alg$, and we have
\begin{equation}\label{series for D_sf}
(\exp^t)'_s(q)= e^{t \re(q)} \sum_{n=0}^\infty \frac{t^{2n+1} \im(q)^{2n}}{(2n+1)!}
   \in \erre
  \qquad \forall q \in Q_\alg \setmeno \erre
\end{equation}
and $(\exp^t)'_s(q) = te^{t\re(q)}$ for every $q \in \erre$. In particular $(\exp^t)'_s$ is a real-valued. By abuse of notation, we write $\exp'_s(t,q)$ to indicate the element $(\exp^t)'_s(q)$ of $\alg$ for every $t \in \erre$ and for every $q \in Q_\alg$, respectively.
\end{Lem}
\begin{proof}
For every $q \in Q_\alg \setmeno \erre$ there exists $\j \in \su_\alg$ and $a, b \in \erre$ such that $b > 0$ and
$q = a + b\j$. Hence $q^c = a^c - \j^cb^c = a - b\j$, $\re(q) = a$, $\im(q) = b\j$, and $|\im(q)| = \sqrt{(b\j)(b\j)^c} = b$. Since $\ci_\j$ and $\ci$ are isomorphic real algebras, we find that 
$\exp^t(q) = e^{tq} = e^{ta}(\cos(tb) + \sin(tb)\j)$ and 
$\exp^t(q^c) = e^{tq^c} = e^{ta}(\cos(tb) - \sin(tb)\j) = (e^{tq})^c$, therefore
\begin{equation}\label{D_sf exp using sin}
(\exp^t)'_s(q) = (q-q^c)^{-1}(e^{tq} - e^{tq^c}) = e^{t\re(q)}\sin(t|\im(q)|)|\im(q)|^{-1} 
  \qquad \forall q \in Q_\alg \setmeno \erre,
\end{equation}
which proves the first equality in \eqref{series for D_sf}. The right-hand side of \eqref{D_sf exp using sin} 
immediately extends by continuity to $te^{ta}$ for $q \in \erre$, thus $(\exp^t)'_s$ is a real-valued function. As $|\im(q)|^2=-\im(q)^2$, by
\eqref{D_sf exp using sin} we have 
\[
(\exp^t)'_s(q) = 
  e^{t\re(q)} \sum_{n=0}^\infty \frac{(-1)^n t^{2n+1}|\im(q)|^{2n}}{(2n+1)!} =e^{t\re(q)} \sum_{n=0}^\infty \frac{t^{2n+1} \im(q)^{2n}}{(2n+1)!}
  \qquad \forall q \in Q_\alg \setmeno \erre,
\]
which implies the second equality of \eqref{series for D_sf}.
\end{proof}

\begin{Cor}
Under the assumption of Theorem \ref{thm:main}, for every $q \in Q_\alg$ such that $\re(q) > \omega$ we have
\begin{equation}
  \Q_q(\A)x = -\int_0^\infty \T(t)\exp'_s(-t,q)\, x \de t \qquad \forall x \in X.
\end{equation} 
\end{Cor}


\section{Integral representation of the powers of $\Q_q(\A)$}\label{Qq^n}

In this section we look for an integral representation of the integer powers of the 
spheri\-cal quasi-resolvent operator $\Q_q(\A)$. In order to find this representation we need the following lemma.

\begin{Lem}\label{convolution}
If $f,g \in L_\T(\clsxint{0,\infty};\alg)$ and $f$ is real-valued, then
\[
\Lap(f)\Lap(g)x = \Lap(f \convstar g)x = \Lap(g \convstar f)x \qquad \forall x \in X,
\]
where we recall that  $f\convstar g : \clsxint{0,\infty} \funzione \alg$, the convolution of $f$ and $g$, is defined by
$(f \convstar g)(t):=\int_0^t f(t-s)g(s)\de s$. 
\end{Lem}

\begin{proof}
Using the fact that $f$ is real-valued, we see at once that $f\convstar g=g\convstar f$. In addition, bearing in mind the semigroup law for $\T$, we find 
\begin{align*}
 \Lap(f)\Lap(g)x 
   & = \int_0^\infty \T(t)f(t)\int_0^\infty\T(s) g(s) x \, \de s \de t  \\
   & = \int_0^\infty \int_0^\infty \T(t)\T(s) f(t)g(s) x \, \de s \de t  \\
   & = \int_0^\infty \int_0^\infty \T(t+s) f(t)g(s) x \, \de s \de t\,.
\end{align*}   
Thus a change of variable and an application of Fubini theorem yields
\begin{align*}
   \Lap(f)\Lap(g)x 
   & = \int_0^\infty \int_t^\infty \T(s) f(t)g(s - t) x \, \de s \de t  \\
   & = \int_0^\infty \int_0^\infty \chi_{\clint{0,s}}(t)\T(s) f(t)g(s - t) x \, \de s \de t  \\
   & = \int_0^\infty \int_0^\infty \chi_{\clint{0,s}}(t)\T(s) f(t)g(s - t) x \,  \de t \de s \\
   & = \int_0^\infty \int_0^t \T(s) f(t)g(s - t) x \,  \de t \de s \\
   & = \int_0^\infty \T(s) \int_0^t  f(t)g(s - t) x \,  \de t \de s \\
   & = \Lap(f\convstar g) = \Lap(g\convstar f).\notag
\end{align*}
The proof is complete.
\end{proof}

Given $n\in\enne\setmeno\{0\}$ and $f \in L_\T(\clsxint{0,\infty};\alg)$, we define $f^{\convstar n}\in L_\T(\clsxint{0,\infty};\alg)$ by
\[
f^{\convstar n} := 
\underbrace{f\convstar f\convstar\,\cdots\,\convstar f}_{\text{$n$ times}}.
\]

\begin{Cor}
Let $\T:\clsxint{0,\infty} \funzione \Lin^r(X)$ be a strongly continuous right linear semigroup, let 
$\A:D(\A) \funzione X$ be its generator, and let $\omega\in\erre$ be a real constant such that $M:=\sup_{t\in\clsxint{0,\infty}}\norma{\T(t)}{}e^{-\omega t}<\infty$.
Given any $q \in Q_\alg$ with $\re(q)>\omega$, we have that $q \in \rho_\s(\A)$ and
\begin{equation}\label{Q^n =}
  \Q_q(\A)^n x = (-1)^n \int_0^\infty \T(t)\exp'_s(-t,q)^{\convstar n}x \de t
\end{equation}
where $\exp'_s(-t,q)^{\convstar n}\in\alg$ indicates the value of $((\exp^{-t})'_s)^{\convstar n}$ at $q$.

Moreover for every $q \in Q_\alg$ with $\re(q)>\omega$ we have
\begin{equation}\label{|Q|^n}
  \norma{\Q_q(\A)^n}{} \le \frac{M}{(\re(q) - \omega)^{2n}} \qquad \forall n \in \enne \setmeno \{0\}.
\end{equation}
\end{Cor}

\begin{proof}
Formula \eqref{Q^n =} follows immediately from $n$ applications of Theorem \ref{thm:main} and Lemma \ref{convolution}. In order to prove \eqref{|Q|^n} let us observe that, given $q\in Q_\alg$, if $a = \re(q)$ and $b = |\im(q)|$, and 
$g(t) = -\exp'_s(-t,q)$ for $t\geq0$, then
\begin{align}
  |(g\convstar g)(t)| 
    & \le \int_0^t |g(t-s)||g(s)| \de s \notag \\
    &  \le    \int_0^t (t-s)e^{-(t-s)a}se^{-sa} \de s \notag \\
    & =  e^{-ta}\int_0^t s(t-s) \de s  
    = \frac{1}{2\cdot 3}t^3e^{-ta}. \notag 
\end{align}
Le us assume by induction that $|g^{\convstar(n-1)}| \le ((2n-3)!)^{-1}t^{2n-3}e^{-at}$. Therefore
for every $n$
\begin{align}
  |g^{\convstar n}(t)|
   & \le \int_0^t |g(t-s)| |g^{\convstar(n-1)}(s)| \de s \notag \\
   & \le \frac{1}{(2n-3)!}\int_0^t (t-s)e^{-(t-s)a} s^{2n-3}e^{-as} \de s \notag \\
   & = \frac{e^{-ta}}{(2n-3)!}\int_0^t (t-s) s^{2n-3} \de s \notag \\
   & = \frac{e^{-ta}}{(2n-3)!}\frac{t^{2n-1}}{(2n-2)(2n-1)} 
     = \frac{t^{2n-1}e^{-ta}}{(2n-1)!}. \notag
\end{align}
Thus
$|g^{\convstar n}(t)| \le \frac{1}{(2n-1)!}t^{2n-1}e^{-ta}$ for every $t \ge 0$ and every 
$n \in \enne \setmeno \{0\}$ and, recalling that $\int_0^\infty t^{2n-1}e^{-t}  \de t = (2n-1)!$, we have
\begin{align}
  \norma{\Q_q(\A)^n}{}
    & \le\int_0^\infty \norma{\T(t)}{} |g^{\convstar n}(t)| \de t  \notag \\
    & \le\int_0^\infty  Me^{\omega t} \frac{1}{(2n-1)!}t^{2n-1}e^{-ta} \de t  \notag \\
    & = \frac{M}{(2n-1)!}\int_0^\infty  t^{2n-1}e^{-(a - \omega)t}  \de t  \notag \\
    & = \frac{M}{(2n-1)!}\int_0^\infty \frac{t^{2n-1}}{(a-\omega)^{2n}}e^{-t}  \de t  \notag \\
        & = \frac{M}{(a - \omega)^{2n}},\notag
\end{align}
and we are done.
\end{proof}




\begin{thebibliography}{40}
\bibitem{Adl95} 
  S. Adler, 
  ``Quaternionic Quantum Field Theory'', 
  Oxford University Press, 1995.


\bibitem{ACFGK2017}
  D. Alpay, F. Colombo, J. Gantner, D.P. Kimsey, \emph{Functions of the
infinitesimal generator of a strongly continuous quaternionic group}, Anal. Appl. (Singap.) {\bf 15} (2017), no. 2, 279--311.

\bibitem{ACK16} 
D. Alpay, F. Colombo, D.P. Kimsey, \emph{The spectral theorem for quaternionic unbounded normal operators based on the S-spectrum}, J. Math. Phys. {\bf 57}, 023503, 27 pp. (2016).


\bibitem{ACS2015}
 D. Alpay, F. Colombo, I. Sabadini,
 \emph{Perturbation of the generator of a quaternionic evolution operator}, Anal. Appl. (Singap.) {\bf 13} (2015), no. 4, 347--370.

\bibitem{AndFul74} 
  F.W. Anderson, K.R. Fuller, 
  ``Rings and Categories of Modules'', 
  Springer Verlag, New York, 1974. 

\bibitem{BirNeu36} 
  G. Birkhoff, J. von Neumann, 
  \emph{The logic of quantum mechanics}, 
  Ann. of Math. (2), \textbf{37} (1936), 823--843. 


\bibitem{ColGanKim18} 
F. Colombo, J. Gantner, D.P. Kimsey, 
``Spectral theory on the S-spectrum for quaternionic operators''. Operator Theory: Advances and Applications. Birkh\"{a}user/Springer, 2018.

\bibitem{ColGan19} 
F. Colombo, J. Gantner,  
``Quaternionic closed operators, fractional powers and fractional diffusion processes''. Operator Theory: Advances and Applications. Birkh\"{a}user/Springer, 2019. 

\bibitem{CoGeSa} 
  F. Colombo, G. Gentili, I. Sabadini, 
  \emph{A Cauchy kernel for slice regular functions}, 
  Ann. Glob. Anal. Geom., \textbf{37} (2010), 361--378.

\bibitem{CoGeSaSt07} 
F. Colombo, G. Gentili, I. Sabadini, D.C. Struppa, \emph{A functional calculus in a non commutative
setting}. Electron. Res. Announc. Math. Sci., \textbf{14} (2007), 60--68.

\bibitem{CoGeSaSt} 
  F. Colombo, G. Gentili, I. Sabadini, D.C. Struppa, 
  \emph{Non Commutative Functional Calculus: Bounded Operators}, Complex Anal. Oper. Theory., \textbf{4} (2010), 821--843. 

\bibitem{CoGeSaSt10} 
  F. Colombo, G. Gentili, I. Sabadini, D.C. Struppa, 
  \emph{Non-commutative functional calculus: Unbounded operators}, 
  J. Geom. Phys., \textbf{60} (2010), 251--259.       

\bibitem{ColSab2009}
 F. Colombo, I. Sabadini,
 \emph{A structure formula for slice monogenic functions and some of its consequences}, Hypercomplex analysis, 101--114, Trends Math., Birkh\"{a}user Verlag, Basel, 2009.

\bibitem{ColSab09} 
  F. Colombo, I. Sabadini, 
  \emph{On some properties of the quaternionic functional calculus}, 
  J. Geom. Anal., \textbf{19} (2009), 601--627.  

\bibitem{ColSab10} 
  F. Colombo, I. Sabadini, 
  \emph{On the formulation of the quaternionic functional calculus}, 
  J. Geom. Phys., \textbf{60} (2009), 1490--1508.    

\bibitem{ColSab11} 
  F. Colombo, I. Sabadini, 
  \emph{The quaternionic evolution operator}, 
  Adv. Math., \textbf{227} (2011), 1772--1805.

  
\bibitem{ColSabStr08} 
  F. Colombo, I. Sabadini, D.C. Struppa, 
  \emph{A new functional calculus for noncommuting operators}, 
 J. Functional Analysis, \textbf{254} (2008), 2255--2274.

\bibitem{CoSaSt}
  F. Colombo, I. Sabadini, D.C. Struppa, 
  ``Noncommutative Functional Calculus'', 
  Birkh\"{a}user, Basel, 2011.

\bibitem{Dav80} 
  E.B. Davies, 
  ``One Parameter Semigroups'', 
  Academic Press, 1980.

\bibitem{numbers} 
   H.-D. Ebbinghaus, H. Hermes, F. Hirzebruch, M. Koecher, K. Mainzer, J. Neukirch, A. Prestel, R. Remmert,
   ``Numbers'', Grad. Texts in Math., vol. 123, Springer-Verlag, New York, 1990.

\bibitem{Emc63} 
  G. Emch, 
  \emph{M\'ecanique quantique quaternionienne et relativit\'e restreinte}, 
  Elv. Phys. Acta, \textbf{36} (1963), 739--769.    

\bibitem{EngNag00} 
  K.-J. Engel, R. Nagel, 
  ``One-Parameter Semigroups for Linear Evolution Equations'', 
  Springer Verlag, New York, 2000.

\bibitem{FinJauSchSpe62} 
  D. Finkelstein, J.M. Jauch, S. Schiminovich, D. Speiser, 
  \emph{Foundations of quaternionic quantum mechanics}, 
  J. Mathematical Phys., \textbf{3} (1962), 207--220.

\bibitem{GenSto12} 
  G. Gentili, C. Stoppato, 
  \emph{Power series and analyticity over the quaternions}, 
  Math. Ann., \textbf{352} (2012), 113-131.

\bibitem{GeSt} 
  G. Gentili, D.C. Struppa, 
  \emph{A new theory of regular functions of a quaternionic variable}, 
  Adv. Math., \textbf{216} (2007), 279--301.

\bibitem{GhiMorPer13} 
  R. Ghiloni, V. Moretti, A. Perotti,
  \emph{Continuous slice functional calculus in quaternionic Hilbert spaces}, 
  Rev. Math. Phys. \textbf{25} (2013), no. 4, 1330006, 83 pp. 

\bibitem{GhiMorPer-bis} 
R. Ghiloni, V. Moretti, A. Perotti,
  \emph{Spectral representations of normal operators in quaternionic Hilbert spaces via intertwining quaternionic PVMs}, 
  Rev. Math. Phys. \textbf{29} (2017), no. 10, 1750034, 73 pp.

\bibitem{GhiPer11} 
  R. Ghiloni, A. Perotti, 
  \emph{Slice regular functions on real alternative algebras}, Adv. Math., \textbf{226} (2011), 1662--1691.


\bibitem{GhiPerRec17} 
 R. Ghiloni, A. Perotti, V. Recupero,
\emph{Noncommutative Cachy integral formula}, 
Complex Anal. Oper. Theory \textbf{11} (2017), 289-306.


\bibitem{GhiRec16} 
R. Ghiloni, V. Recupero, 
  \emph{Semigroups over real alternative *-algebras: Generation theorems and spherical sectorial operators}, Trans. Amer. Math. Soc., \textbf{368} (2016), no. 4, 2645--2678.
  
\bibitem{GhiRec18} 
R. Ghiloni, V. Recupero, 
  \emph{Slice regular semigroups}, Trans. Amer. Math. Soc., \textbf{370} (2018), no. 7, 4993--5032.  

\bibitem{GM91} 
J.E. Gilbert, M.A.M Murray, ``Clifford algebras and Dirac operators in harmonic analysis'', Cambridge University Press, Cambridge, 1991.

\bibitem{Gol85}
  J.A. Goldstein, 
  ``Semigroups of Operators and Applications'', 
  Oxford University Press, 1985.

\bibitem{GHS08}
K. G{\"u}rlebeck, K. Habetha, W. Spr{\"o}{\ss}ig, ``Holomorphic functions in the plane and $n$-dimensional space'', Birkh\"auser Verlag, Basel, xiv+394, 2008.

\bibitem{HilPhi57}
  E. Hille, R.S. Phillips, 
  ``Functional Analysis and Semigroups'',
   Amer. Math. Soc. Coll. Publ., vol. 31, Amer. Math. Soc., 1957.

\bibitem{HorBie84} 
  L. P. Horwitz, L.C. Biedenharn, 
  \emph{Quaternionic quantum mechanics: Second quantization and gauge field}, 
  Annals of Physics, \textbf{157} (1984), 432--488.

\bibitem{Kap} 
I. Kaplansky, 
\emph{Normed algebras}, Duke Mat. J., \textbf{16} (1949), 399--418.    


\bibitem{Lun95}
  A. Lunardi, 
  ``Analytic Semigroups and Optimal Regularity in Parabolic Problems'', 
  Birkh\"{a}user-Verlag, 1995.

  
  
\bibitem{Paz83} 
  A. Pazy, 
  ``Semigroups of Linear Operators and Applications to Partial Differential Equations'', 
  Springer-Verlag, New York, 1983.




\bibitem{Tai95} 
K. Taira, ``Analytic Semigroups and Semilinear Initial Boundary Value Problems'', London Math. Soc. Lect. Notes Ser., vol. 223, Cambridge University Press, 1995.
\end{thebibliography}
\end{document}